\documentclass[11pt]{article}

\usepackage{amsbsy}
\usepackage{amssymb}
\usepackage{amsmath}
\usepackage{amsthm}

\usepackage{times}
\usepackage{bm}
\usepackage{dsfont} 

\usepackage{xcolor}

\usepackage{graphicx}

\usepackage{url}

\usepackage{natbib}
\usepackage[plain,noend]{algorithm2e}

\usepackage[top=1in, bottom=1in, left=1in, right=1in]{geometry}

\newcommand{\E}{\mathrm{E}}
\newcommand{\var}{\text{var}}
\newcommand{\cov}{\text{cov}}

\newcommand{\X}{\mathcal{X}}

\newcommand{\A}{\mathbf{A}}

\newcommand{\bd}{\mathbf{d}}
\newcommand{\Y}{\mathbf{Y}}
\renewcommand{\X}{\mathbf{X}}

\newcommand{\G}{\mathcal{G}}
\newcommand{\V}{\mathcal{V}}
\newcommand{\gE}{\mathcal{E}}

\newcommand{\Ao}{\mathbf{A}^{\!\mathcal{O}}}

\newtheorem{lemma}{Lemma}
\newtheorem{definition}{Definition}
\newtheorem*{defn*}{Definition}
\newtheorem{proposition}{Proposition}
\newtheorem{corollary}{Corollary}

\newtheorem{assumption}{Assumption}

\title{Confidence intervals for means under constrained dependence}

\author{Peter M. Aronow$^{1,2}$, Forrest W. Crawford$^{2}$, and Jos\'e R. Zubizarreta$^{3}$\\[1em]
1. Department of Political Science, Yale University \\
2. Department of Biostatistics, Yale School of Public Health \\ 
3. Division of Decision, Risk and Operations, and Department of Statistics, Columbia University }

\begin{document}

\maketitle 

\begin{abstract}
\noindent We develop a general framework for conducting inference on the mean of dependent random variables given constraints on their dependency graph.  We establish the consistency of an oracle variance estimator of the mean when the dependency graph is known, along with an associated central limit theorem. We derive an integer linear program for finding an upper bound for the estimated variance when the graph is unknown, but topological and degree-based constraints are available.  We develop alternative bounds, including a closed-form bound, under an additional homoskedasticity assumption.  We establish a basis for Wald-type confidence intervals for the mean that are guaranteed to have asymptotically conservative coverage.  We apply the approach to inference from a social network link-tracing study and provide statistical software implementing the approach.\\[1em]
\textbf{Keywords} dependency graph, HIV prevalence, oracle estimator, variance estimate 
\end{abstract}


\section{Introduction}

Researchers often encounter dependent data, where the exact nature of that dependence is unknown, and they wish to make inferences about outcome means.  Current methods typically assume either independence of unit outcomes, or that the dependency structure is known or directly estimable \citep{liang1986longitudinal,conley1999gmm,white2014asymptotic,ogburn2014vaccines,cameron2015practitioner,tabord2015inference}.  In many cases, however, researchers may only have limited information about the nature of dependence between units, or perhaps only the number of other units on which a given unit's outcome depends.  For example, in studies of units embedded in a network, the degrees to which subjects are connected may be known, but the identities of the other subjects to whom they are connected may often remain unobserved \citep[e.g.,][]{crawford2016graphical}. The underlying relationships may be represented by a dependency graph \citep{baldi1989normal}, where vertices represent individual units and edges represent the possibility of probabilistic dependence.  A dependency graph is not a generative graphical model for outcomes, such as a Markov random field. Rather, a dependency graph is a description of possible non-independence relationships between units.

In this paper, we develop a framework for constructing confidence intervals for the mean of dependent random variables, where their dependency graph is unknown or partially known but subject to topological constraints.  Considering the class of Wald-type normal-approximation-based estimators given the sample mean, we seek an upper bound for the estimated variance of the sample mean using upper bounds for the degrees of each unit in the dependency graph and a local dependence assumption.  We show that this optimization problem can be expressed as a integer linear program for the elements of the dependency graph adjacency matrix. 
We implement this approach in the new statistical software package \texttt{depinf} for \texttt{R}.
The approach may be used even when no edges in the dependency graph are known. We also derive more computationally simple bounds, including a closed-form bound, when the random variables are assumed to be homoskedastic.  We illustrate the utility of the method using data from a social link-tracing study of individuals at high risk for HIV infection in St. Petersburg, Russia.


\section{Setting}

Consider a simple undirected graph $\G=(\V,\gE)$ with no parallel edges or self-loops. Let $|\V|=N$.  Associated with each vertex $i\in \V$ is a random variable $X_i$, and $\G$ characterizes probabilistic dependencies in the outcomes \citep[e.g.,][]{baldi1989normal}.
\begin{definition}[Dependency graph]
  $\G$ is a dependency graph if for all disjoint sets $\V_1,\V_2\subset \V$ with no edge in $\gE$ connecting a vertex in $\V_1$ to a vertex in $\V_2$, the set $\{X_i:\ i\in\V_1\}$ is independent from the set $\{X_j:\ j\in\V_2\}$.
  \label{defn:depgraph}
\end{definition}
We emphasize that a dependency graph represents a set of possible non-independence relationships among units, not a graphical model that induces dependencies. 

Suppose $\G$ is a dependency graph and we observe a subset $\V_S\subseteq\V$, where $|\V_S|=n$.  Label these observed vertices $1,\ldots,n$, and label the unobserved vertices in $\V\setminus\V_S$ arbitrarily by $n+1,\ldots,N$.  For each $i\in \V_S$, we observe the outcomes $X_1,\ldots,X_n$ and the degrees $d_i=|\{j:\ \{i,j\}\in \gE\}|$ for each $i\in\V_S$.  
\begin{definition}[Induced subgraph]
  For a set of vertices $\V_S\subseteq\V$, the induced subgraph in $\G$ is $\G_S=(\V_S,\gE_S)$, where $\gE_S = \{ \{i,j\}: i\in\V_S,\ j\in\V_S, \text{ and } \{i,j\}\in\gE\}$.
  \label{defn:induced}
\end{definition}
Let $\G_S=(\V_S,\gE_S)$ be the induced subgraph of the observed vertices $\V_S$. It follows that $\G_S$ is also a dependency graph. Let $\G_R=(\V_S,\gE_R)$ be a subgraph of $\G_S$, consisting of all the observed vertices in $\V_S$, and a subset of the edges in $\gE_S$.  
\begin{assumption}[Observed data]
  We observe the outcomes $X_1,\ldots,X_n$, the degrees $d_1,\ldots,d_n$, and $\G_R$.
  \label{assump:obs}
\end{assumption}
Let $\X=(X_1,\ldots,X_n)$, $\bd=(d_1,\ldots,d_n)$, and denote the observed data as $\Y=(\X,\bd,\G_R)$.  

We wish to conduct inference on the mean $\mu = \frac{1}{n} \sum_{i \in V_S} \E[X_i]$ given $\Y$. The mean $\mu$ is a functional of the joint distribution of outcomes for the units in the sample, and is accordingly a data-adaptive target parameter \citep{vanderlaan2013statistical,balzer2015targeted} and not necessarily a feature of any broader population of units. Let $\overline{X}=n^{-1}\sum_{i\in\V_S} X_i$. We proceed by constructing conservative estimators of 
\[ \var(\overline{X}) = \frac{1}{n^2} \sum_{i\in\V_S}^n \sum_{j\in\V_S}^n \cov(X_i,X_j). \]
We may use the square roots of these estimates as standard error estimators in order to construct Wald-type confidence intervals about the sample mean that are guaranteed to have asymptotic coverage for $\mu$ at greater than or equal to nominal levels.   


\section{Variance estimation}

The observed subgraph $G_R$ may not reveal all the edges in $G_S$ that connect observed vertices.  We consider a class of variance estimators that depend on knowledge of $G_S$, whose structure is represented by an $n\times n$ binary symmetric adjacency matrix in which rows and columns are ordered by the indices $1,\ldots,n$ of the vertices in $V_S$.  We now define some key concepts.
\begin{definition}[Compatibility]
  The $n\times n$ binary symmetric adjacency matrix $\A$ is compatible with the observed data $\Y$ if for each $\{i,j\}\in \gE_R$, $\A_{ij}=\A_{ji}=1$, and for each $i\in\V_S$, $\sum_{j\in \V_S} \A_{ij} \le d_i$.
  \label{defn:compatibility}
\end{definition}
The last condition in Definition \ref{defn:compatibility} requires that the degree of $i$ in the subgraph $\G_S$ not be greater than its degree in the full graph $\G$.  Let $\Ao=\{\Ao_{ij}\}$ be the true $n\times n$ adjacency matrix of $G_S$, where $\Ao_{ij}=1$ if $\{i,j\}\in \gE_S$ for $i,j\in \V_S$ and 0 otherwise.  Let $\mathcal{A}(\Y) = \{\A: \A\text{ is compatible with } \Y\}$ in the sense of Definition \ref{defn:compatibility}; it is clear that $\Ao\in \mathcal{A}(\Y)$. 
\begin{definition}[Oracle estimator]
  For a family of variance estimators $\widehat{V}(\A;\Y)$ defined for $\A\in\mathcal{A}(\Y)$, the oracle estimator is $\widehat{V}(\Ao;\Y)$.
\end{definition}
For a variance estimator $\widehat{V}(\A;\Y)$, define the set $\mathcal{A}^m = \{ \A\in\mathcal{A}(\Y):\  \widehat{V}(\A;\Y) \text{ is maximized}\}$.
\begin{definition}[Maximal compatible estimator]
  Let $\A^m\in \mathcal{A}^m$. The maximal compatible estimator is $\widehat{V}(\A^m;\Y)$.
\end{definition}
The maximal compatible estimator provides a sharp upper bound for the oracle estimator because $\widehat{V}(\Ao;\Y) \le \widehat{V}(\A^m;\Y)$. Finally, define the plug-in sample variance, $\hat\sigma^2 = n^{-1} \sum_{i\in \V_S} (X_i - \overline{X})^2$. 

We now describe an asymptotic scaling, along with boundedness conditions for outcome values and unit degrees. In particular, bounding degrees suffices to ensure sufficient sparsity in the dependency graph to allow for root-$n$ consistency, a central limit theorem, and convergence of the variance estimator.
\begin{assumption}[Asymptotic scaling] 
Consider the sequence $(\G,\Y)_n$ of nested graphs $\G$ and observed data $\Y=(\G_R,\X,\bd)$, where $\G_R=(\V_S,\E_R)$, $|\V_R|=n$, and $|\V| = N_n \geq n$.  Assume there exist finite, positive constants $c_1$, $c_2$ such that for every element $(\G,\Y)_n$, 
$\Pr( |X_i - \mu | > c_1) = 0, \forall i \in \V_S$ (bounded outcome values) and
$\sum_{j\in \V_S} \Ao_{ij} \leq c_2, \forall i \in \V_S$ (bounded degrees in the dependency graph). Further assume there exists a finite, positive constant $c_3$ such that $\lim_{n \rightarrow \infty} n \var(\overline X) = c_3$ (nondegenerate limiting variance).
\label{assump:scaling}
\end{assumption}
We will proceed by deriving oracle estimators under two sets of nested assumptions. We establish their asymptotic properties, then derive feasible estimators that 
dominate the oracle estimators. 


\subsection{General Case}

\label{sec:general}

We first consider the case where we impose no distributional assumptions on the distribution of any $X_i$ (beyond the boundedness conditions of Assumption \ref{assump:scaling}).  Define the estimator
\begin{equation}	\label{eq:v1}
  \widehat{V}_1(\A;\Y) = \frac{1}{n^2} \left[ n\hat\sigma^2 + \sum_{i\in\V_S} \sum_{j\in\V_S} \A_{ij} (X_i - \overline{X}) (X_j - \overline{X}) \right].
\end{equation}
The corresponding oracle estimator $\widehat{V}_1(\Ao;\Y)$ is consistent.
\begin{proposition}
Under Assumption \ref{assump:scaling}, for any $\epsilon > 0$,
\[ \lim_{n\to\infty} \Pr(| n \widehat{V}_1(\Ao;\Y) - n \var(\overline X) | > \epsilon) = 0.  \]
\label{prop:V1consistent}
\end{proposition}
\begin{proof} 
  We follow the general proof strategy of \citet{aronow2013estimating}.    We will establish mean square convergence of $n \widehat{V}_1(\Ao;\Y)$ to $n \var(\overline X)$, allowing us to invoke Chebyshev's inequality to prove the proposition.  
  Decompose $\hat\sigma^2 = n^{-1}\sum_{i=1}^n X_i^2 - n^{-2}\left(\sum_{i=1}^n X_i\right)^2$. 
  Linearity of expectations implies $\E[\overline{X}] = \mu$ and $\E[ \overline{X^2}] = n^{-1} \sum_{i=1}^n \E[X_i^2]$. 
  Since Assumption \ref{assump:scaling} guarantees bounded outcomes, and the number of nonzero elements in the covariance matrix of outcome values is $O(n)$, $\var(\overline{X}) = O(n^{-1})$ and $\var(\overline{X^2}) = O(n^{-1})$, yielding convergence of $\hat{\sigma}^2$. 

Next we address convergence of the second term $n^{-1}\sum_{i\in\V_S} \sum_{j\in\V_S} \Ao_{ij} (X_i - \overline{X}) (X_j - \overline{X})$. Asymptotic unbiasedness follows directly from linearity of expectations and $\var(\overline{X}) = O(n^{-1})$. To establish mean square convergence, we consider the variance 
\begin{equation}
  \begin{split}
    &\var\left( \frac{1}{n} \sum_{i\in V_S} \sum_{j\in V_S} \Ao_{ij} (X_i - \overline{X})(X_j - \overline{X})\right) \\
    &\quad = \frac{1}{n^2} \sum_{i,j,k,l\in V_S} \cov\left(\Ao_{ij}(X_i-\overline{X})(X_j-\overline{X}), \Ao_{kl}(X_k-\overline{X})(X_l-\overline{X})\right)  \\
    &\quad = \frac{1}{n^2} \sum_{i,j,k,l\in V_S} \Ao_{ij}\Ao_{kl} \cov\left((X_i-\overline{X})(X_j-\overline{X}), (X_k-\overline{X})(X_l-\overline{X})\right) 
  \end{split}
  \label{eq:var2term}
\end{equation}
where the last line follows from bilinearity of covariance. Letting 
\[ \xi_{ijkl} = \cov\big((X_i-\overline{X})(X_j-\overline{X}), (X_k-\overline{X})(X_l-\overline{X})\big), \]
we now examine the conditions under which $\xi_{ijkl}\neq 0$. Expanding the covariance, 
\begin{equation}
  \begin{split}
    \xi_{ijkl} &= \cov\big((X_i-\overline{X})(X_j-\overline{X}), (X_k-\overline{X})(X_l-\overline{X})\big)  \\
    &= \E\big[(X_i-\overline{X})(X_j-\overline{X})(X_k-\overline{X})(X_l-\overline{X})\big] \\
    &\quad -  \E\big[(X_i-\overline{X})(X_j-\overline{X})\big]\E\big[(X_k-\overline{X})(X_l-\overline{X})\big]  \\
    &= \E[X_iX_jX_kX_l] 
 - \E[X_iX_jX_k\overline{X}]  
 - \E[X_iX_jX_l\overline{X}] 
 - \E[X_iX_kX_l\overline{X}] \\
 &\quad - \E[X_jX_kX_l\overline{X}]  
 + \E[X_iX_j\overline{X}^2] 
 + \E[X_iX_k\overline{X}^2] 
 + \E[X_iX_l\overline{X}^2] \\
 &\quad + \E[X_jX_k\overline{X}^2]  
 + \E[X_jX_l\overline{X}^2] 
 + \E[X_kX_l\overline{X}^2] \\
 &\quad - \E[X_i\overline{X}^3] 
  - \E[X_j\overline{X}^3] 
 - \E[X_k\overline{X}^3]
 - \E[X_l\overline{X}^3] 
  + \E[\overline{X}^4] \\
 &\quad - \big[\E[X_iX_j]\E[X_kX_l]
 - \E[X_iX_j]\E[X_k\overline{X}]
 - \E[X_iX_j]\E[X_l\overline{X}] \\
 &\quad + \E[X_iX_j]\E[\overline{X}^2] 
        - \E[X_i\overline{X}]\E[X_kX_l]  
        + \E[X_i\overline{X}]\E[X_l\overline{X}] \\
 &\quad + \E[X_i\overline{X}]\E[X_k\overline{X}] 
        - \E[X_i\overline{X}]\E[\overline{X}^2]
        - \E[X_j\overline{X}]\E[X_kX_l]  \\
 &\quad + \E[X_j\overline{X}]\E[X_l\overline{X}] 
        + \E[X_j\overline{X}]\E[X_k\overline{X}] 
        - \E[X_j\overline{X}]\E[\overline{X}^2] \\
 &\quad + \E[\overline{X}^2]\E[X_kX_l] 
        - \E[\overline{X}^2]\E[X_k\overline{X}]
        - \E[\overline{X}^2]\E[X_l\overline{X}]
        + \E[\overline{X}^2]\E[\overline{X}^2] \big]
 \end{split}
 \label{eq:cijkl}
\end{equation}
Then by root-$n$ consistency of means and Slutsky's Theorem, as $n\to\infty$ expectations involving $\overline{X}$ factorize, yielding, e.g. $\E(X_i\overline{X}) = \E(X_i)\mu + O(n^{-1})$.  We therefore combine terms and rewrite \eqref{eq:cijkl} as
\begin{equation}
  \begin{split}
    \xi_{ijkl} &= \cov(X_iX_j,X_kX_l) \\
    & - \mu\big(\cov(X_iX_j,X_k) + \cov(X_iX_j,X_l) + \cov(X_i,X_kX_l) + \cov(X_jX_kX_l)\big) \\
    & + \mu^2\big(\cov(X_i,X_k) + \cov(X_i,X_l) + \cov(X_j,X_k) + \cov(X_j,X_l) \big) + O(n^{-1}) \\ 
    & = \xi'_{ijkl} + O(n^{-1}),
 \end{split}
 \label{eq:cijkl2}
\end{equation}
where the limiting covariance is denoted $\xi'_{ijkl}$.  This can only be nonzero if at least one of the covariance terms in \eqref{eq:cijkl2} is nonzero.  Since $\G_S$ is a dependency graph, this condition is only met when there exists at least one edge between a vertex in the set $\{i,j\}$ and a vertex in the set $\{k,l\}$.  Therefore $\Ao_{ij}\Ao_{kl} \xi'_{ijkl}$ can only be nonzero if 
\[ \{\Ao_{ij}=\Ao_{kl}=1\} \text{ and } \left( \{\Ao_{ik}=1\} \text{ or } 
  \{\Ao_{il}=1\} \text{ or }
  \{\Ao_{jk}=1\} \text{ or }
\{\Ao_{jl}=1\} \right). \]
By Assumption \ref{assump:scaling}, the degree of each vertex in $V_S$ is bounded by $c_2$, so the condition is satisfied by at most $4nc_2^3$ terms in the summation in \eqref{eq:var2term}. In addition, we may compute the remainder term $\sum_{i,j,k,l\in V_S} \Ao_{ij}\Ao_{kl}(\xi_{ijkl} - \xi'_{ijkl}) = \sum_{i,j,k,l\in V_S} \Ao_{ij}\Ao_{kl} O(n^{-1}) = O(n)$, thus both terms are $O(n)$ before dividing by $n^2$. Therefore 
$\var\left( n^{-1} \sum_{i\in V_S} \sum_{j\in V_S} \Ao_{ij} (X_i - \overline{X})(X_j - \overline{X})\right) = O(n^{-1})$
and the result follows.
\end{proof}

Proposition 1 is readily applicable to problems where the dependency graph is known, as it provides a basis for consistent variance estimation, generalizing results for special cases \citep{conley1999gmm,aronow2015cluster}. We now address the case where the true subgraph $\G_S$ is not known, but constraints on the graph are available.

Let $\mathcal{A}_1^m = \{ \A\in\mathcal{A}(\Y):\  \widehat{V}_1(\A;\Y) \text{ is maximized}\}$ be the set of compatible adjacency matrices that maximize $\widehat{V}_1(\A;\Y)$.  We can find an element $\A^m$ of $\mathcal{A}_1^m$ by solving the 0-1 integer linear program
\begin{equation} \label{eq:ilp1}
\begin{aligned}
& \underset{\A}{\text{maximize}}
&& (\X-\overline \X)' \A (\X-\overline \X) \\
& \text{subject to}
&& \A \mathbf{1} \preceq \mathbf{d}, \\
&&& \A \succeq \A_R,
\end{aligned}
\end{equation}
where $\A_R$ is the adjacency matrix of $\G_R$ and $\preceq$ denotes the element-wise ``less-than'' relation.
Since $\A$ is an adjacency matrix, we can reduce the program and maximize over the decision variables that correspond to the upper or lower diagonal elements of $\A$ only (for details, see the supplementary materials).
The resulting program has $n(n - 1)/2$ decision variables and in general it is a multidimensional knapsack problem \citep{kellerer2004introduction}.  
In the abstract, this problem is NP-hard problem, but it admits a polynomial time approximation scheme (PTAS).
Nonetheless, typical PTAS depend heavily on the size of the problem and their running time is very high (see, e.g., section 9.4.2 of \citealt{kellerer2004introduction}). 
In spite of this, in standard practice, for example with 1000 observations or less as in our application in Section \ref{sec:app}, problem (\ref{eq:ilp1}) can be solved in a few seconds with modern optimization solvers such as Gurobi.  
To obtain a solution within a provably small optimality gap, these solvers use a variety of techniques, including: linear programming and branch-and-bound procedures to reduce the set of feasible solutions; presolve routines applied prior to the branch-and-bound procedures to reduce the size of the problem; cutting planes methods to remove fractional solutions and tighten the formulation; and a collection of heuristics to find good incumbent solutions in the branch-and-bound \citep{bixby2007progress, linderoth2012milp, nemhauser2013impact}.  
All these techniques are used in parallel by exploiting the availability of multiple cores in computers today.  
We provide an implementation in the new statistical package \texttt{depinf} for \texttt{R}.

While the true adjacency matrix $\Ao$ is not known, an element $\A^m \in \mathcal{A}_1^m$ produces a variance estimate $\widehat{V}_1(\A^m,\Y)$ that is at least as large as the oracle estimator $\widehat{V}_1(\Ao;\Y)$.  As $n$ grows large, the variance estimate $\widehat{V}_1(\A^m,\Y)$ is conservative: the probability that $n \widehat{V}_1(\A^m)$ underestimates $n \var(\overline X)$ by more than $\epsilon>0$ tends to zero. 
\begin{corollary}
Given Assumption \ref{assump:scaling}, then for any $\epsilon > 0$,
\[ \lim_{n\to\infty} \Pr(n\var(\overline{X}) - n\widehat{V}_1(\A^m;\Y) > \epsilon) = 0.  \]
\label{cor:conservative}
\end{corollary}
\begin{proof}
  Across all sample realizations, $\widehat{V}_1(\A^m;\Y) \geq \widehat{V}_1(\Ao;\Y)$. Then 
\begin{equation}
\begin{split}
   &\lim_{n\to\infty} \Pr(n\var(\overline{X}) - n\widehat{V}_1(\A^m;\Y) > \epsilon) \\
    &\le \lim_{n\to\infty} \Pr(n\var(\overline{X}) - n\widehat{V}_1(\A^m;\Y) + n\widehat{V}_1(\A^m;\Y) - n\widehat{V}_1(\Ao;\Y) > \epsilon) \\
    &= \lim_{n\to\infty} \Pr(n\var(\overline{X}) - n\widehat{V}_1(\Ao;\Y) > \epsilon) \\
    &= 0  
\end{split}
\end{equation}
by Proposition \ref{prop:V1consistent}.
\end{proof}
Corollary \ref{cor:conservative} does not imply consistency of $\widehat{V}_1(\A^m;\Y)$ as an estimator of $n\var(\overline{X})$, nor does it imply that the estimator converges to any particular limiting value. Rather we have established that, for large $n$, its distribution will tend to be at least as large as the true variance.


\subsection{Alternative bounds under homoskedasticity}

\label{sec:homoskedastic}

When all variances are equal, we can obtain alternative closed-form bounds that are computationally simpler and is less sensitive to between-sample variability in the empirical variance-covariance matrix. This estimator essentially only depends on the estimated variance of unit outcomes and the maximum number of edges in the dependency graph.
\begin{assumption}[Homoskedasticity]
$\var(X_i) = \var(X_j), \forall i,j\in\V$.
\label{assump:homoskedastic}
\end{assumption}
Under homoskedasticity, the general estimator $\widehat{V}_1(\A^m,\Y)$ developed in Section \ref{sec:general} provides conservative variance estimate.  A bound that is relatively computationally simple to compute can be derived by noting that when $\var(X_i)=\sigma^2$, $\cov(X_i,X_j) \le \sigma^2 \Ao_{ij}$.  To this end, define the estimator 
\begin{equation}	\label{eq:v2}
    \widehat{V}_2(\A;\Y)   = \frac{\hat\sigma^2}{n}\left[ 1+ \frac{1}{n}\sum_{i\in \V_S}\sum_{j\in \V_S} \A_{ij} \right].
\end{equation}
The oracle estimator $\widehat{V}_2(\Ao,\Y)$ is not generally consistent, though it is asymptotically conservative.
\begin{proposition}
Given Assumptions \ref{assump:scaling} and \ref{assump:homoskedastic}, then for any $\epsilon > 0$,
\[ \lim_{n\to\infty}  \Pr(n\var(\overline X) - n\widehat{V}_2(\Ao;\Y) > \epsilon) = 0.  \]
\end{proposition}
\begin{proof}
To prove the claim, we first define an alternative oracle estimator which presumes knowledge of the $\rho_i$ values, $$\widehat{V}_2^*(\Ao;\Y) =  \frac{\hat\sigma^2}{n}\left[ 1+ \frac{1}{n}\sum_{i\in \V_S}\sum_{j\in \V_S} \Ao_{ij} \rho_i \right].$$ Multiplying by $n$, 
$n\widehat{V}_2^*(\Ao;\Y) =  {\hat\sigma^2}\left[ 1+ \frac{1}{n}\sum_{i\in \V_S}\sum_{j\in \V_S} \Ao_{ij} \rho_i \right].$ As in the proof of Proposition 1,  $\hat\sigma^2$ converges in mean square. By Assumption 2, $1 \leq 1+ \frac{1}{n}\sum_{i\in \V_S}\sum_{j\in \V_S} \Ao_{ij} \leq 1 + c_2$, allowing us to invoke Slutsky's Theorem and Chebyshev's Inequality to show $\lim_{n\to\infty}  \Pr(|n \widehat{V}_2^*(\Ao;\Y) - n \var(\overline X)| < \epsilon) = 0.$ The Cauchy-Schwarz Inequality (i.e., all $\rho_i \leq 1$) implies $\widehat{V}_2^*(\Ao;\Y) \leq \widehat{V}_2(\Ao;\Y)$ across all sample realizations. The result follows directly.
\end{proof}

As before, we can maximize the estimator $\widehat{V}_2(\A;\Y)$ over the family of compatible graphs. Define $\mathcal{A}_2^m = \{\A\in\mathcal{A}(\Y):\  \widehat{V}_2(\A;\Y) \text{ is maximized}\}$, and let $A^m\in\mathcal{A}_2^m$. To find an element of $\mathcal{A}_2^m$, we solve the 0-1 integer linear program
\begin{equation} \label{eq:ilp2} 
\begin{aligned}
& \underset{\A}{\text{maximize}}
&& \mathbf{1}' \A \mathbf{1} \\
& \text{subject to}
&& \A \mathbf{1} \preceq \mathbf{d}, \\
&&& \A \succeq \A_R,
\end{aligned}
\end{equation}
where again $\A$ is an arbitrary 0-1 adjacency matrix and $\A_R$ is the adjacency matrix of $\G_R$. Note that  finding the solution to this problem does not depend on the empirical variance-covariance matrix; the variability of the estimator $\widehat{V}_2(\A^m;\Y)$ is purely attributable to estimation error in $\hat\sigma^2$.

Since $\widehat{V}_2(\A;\Y)$ does not rely on any feature of $\A$ other than the number of positive entries, we can derive a looser closed-form upper bound by considering the maximum number of edges that can be in $\E_S$.  For $i\in \V_S$, let $d_i' = \min\{d_i,n-1\}$ be the degree of $i$ in $\G$, truncated at $n-1$.  Let 
\begin{equation}
  \widehat{V}_2'(\Y) = \frac{\hat\sigma^2}{n}\left[ 1+ \frac{1}{n}\sum_{i\in \V_S} d'_i \right] .
  \label{eq:V2p}
\end{equation}
The estimator \eqref{eq:V2p} does not depend on any particular member of the set $\mathcal{A}$ of compatible adjacency matrices.  

\begin{lemma}
  We have $\widehat{V}_2(\Ao,\Y) \le \widehat{V}_2(\A^m;\Y) \le  \widehat{V}_2'(\Y)$, 
  with $\widehat{V}_2(\A^m;\Y) = \widehat{V}_2'(\Y)$ when there exists a compatible adjacency matrix $\A^m\in\mathcal{A}$ such that $d_i' = \sum_{j\in V_S} \A_{ij}^m$.
  \label{lem:V2}
\end{lemma}
\begin{proof}
  By definition, $\widehat{V}_2(\A;\Y) \le \widehat{V}_2(\A^m;\Y)$ for every $\A\in\mathcal{A}$.  Since $\Ao\in \mathcal{A}$, it follows that $\widehat{V}_2(\Ao,\Y) \le \widehat{V}_2(\A^m;\Y)$.  Now let $d_i^m=\sum_{j\in \V_S} \A^m_{ij}$ be the degree of $i$ in the adjacency matrix $\A^m$, and note that for every $i\in \V_S$, $d_i^m \le d_i'$.  Then 
  \begin{equation*}
    \begin{split}
      \widehat{V}_2(\A^m;\Y) &= \frac{\hat\sigma^2}{n}\left[ 1+ \frac{1}{n}\sum_{i\in \V_S}\sum_{j\in \V_S} \A_{ij}^m \right] \\
      &= \frac{\hat\sigma^2}{n}\left[ 1+ \frac{1}{n}\sum_{i\in \V_S} d_i^m \right] \\
      &\le \frac{\hat\sigma^2}{n}\left[ 1+ \frac{1}{n}\sum_{i\in \V_S} d_i' \right] \\
      &= \widehat{V}_2'(\Y)
    \end{split}
  \end{equation*}
  as claimed.  Now consider a compatible adjacency matrix $\A\in \mathcal{A}$ with the property that $d_i'=\sum_{j\in\V_S}\A_{ij}$.  From the program \eqref{eq:ilp2} we see that $\mathbf{1}'\A\mathbf{1} = \sum_{i\in\V_S} d_i'$ is two times the maximal number of edges in $\G_S$, $\A\mathbf{1} = \bd' \preceq \bd$ by the definition of $\bd'=(d_1',\ldots,d_n')$, and $\A \succeq \A_R$ since $\A\in\mathcal{A}$.  It follows that $A\in \mathcal{A}^m$, so we may call $\A^m=\A$. Therefore $\widehat{V}_2(\A^m;\Y)=\widehat{V}_2'(\Y)$, as claimed.
\end{proof}
Lemma \ref{lem:V2} implies a simple, conservative correction to the variance under homoskedasticity; simply multiply the conventional variance estimate $\frac{\hat\sigma^2}{n}$ by $1 + \overline {d'}$, where $\overline {d'}$ is the average truncated degree.

As expected, the upper bound estimators under homoskedasticity are asymptotically conservative.
\begin{corollary}
Given Assumptions \ref{assump:scaling} and \ref{assump:homoskedastic}, then for any $\epsilon > 0$,
\begin{align*}
 & \lim_{n\to\infty} \Pr(n \var(\overline{X}) - n\widehat{V}_2(\A^m;\Y) > \epsilon) = 0, \\ 
 & \lim_{n\to\infty} \Pr(n \var(\overline X) - n\widehat{V}_2'(\Y) > \epsilon) = 0.
 \end{align*}
\end{corollary}
The proof follows from Lemma \ref{lem:V2} and the same reasoning employed in the proof of Corollary \ref{cor:conservative}.


\section{Wald-type confidence intervals}

We now prove that our variance estimates can be used to form valid Wald-type confidence intervals about the sample mean. First, we establish a central limit theorem for the sample mean given our asymptotic scaling.
\begin{lemma}
Given Assumption \ref{assump:scaling},
\[ \left(\overline X -\mu \right)\bigg/\sqrt{\var(\overline X)} \rightarrow_{d} N(0, 1) \].
\end{lemma}
Lemma 2, a standard result in applying Stein's method to the setting of local dependence, has been proven by, e.g., Theorem 2.7 of \citet{chen2004normal}. Similarly, we reiterate the well-known basis for Wald-type confidence intervals.

\begin{lemma}
  Given Assumption \ref{assump:scaling}, if a variance estimator $\widehat{V}(\A;\Y)$ satisfies 
  \[ \lim_{n\to\infty}\Pr(| n \widehat{V}(\A;\Y) - n \var(\overline X) | > \epsilon) = 0, \] 
  then confidence intervals formed as $\overline X \pm z_{1-\alpha/2} \sqrt{\widehat{V}(\A;\Y)}$ will have $100(1-\alpha)\%$ coverage for $\mu$ in large $n$.
  \label{lem:wald}
\end{lemma}
Lemma 3 follows directly from Lemma 2 and Slutsky's Theorem.

We now establish the validity of confidence intervals constructed via Lemma \ref{lem:wald}. 
\begin{proposition}
  Given Assumption \ref{assump:scaling}, if a variance estimator $\widehat{V}(\A;\Y)$ satisfies 
  \[ \lim_{n\to\infty}\Pr( n\var(\overline{X}) -  n\widehat{V}(\A;\Y) > \epsilon) = 0, \]
  then confidence intervals formed as $\overline X \pm z_{1-\alpha/2} \sqrt{\widehat{V}(\A;\Y)}$ will have at least $100(1-\alpha)\%$ coverage for $\mu$ in large $n$.
\end{proposition}
\begin{proof}
Define a random variable $U$ such that 
\[ U = \begin{cases} 
    \widehat{V}(\A;\Y) & \text{if } \widehat{V}(\A;\Y) \le \var(\overline X)  \\
    \var(\overline{X})  & \text{otherwise}. 
  \end{cases}
  \]
  Then $\lim_{n\to\infty}\Pr(|nU - n\var(\overline X)| > \epsilon) = 0$, and by Lemma 3 Wald-type confidence intervals formed with $U$ as a variance estimate will have at least proper coverage. Across every sample realization, $\widehat{V}(\A;\Y) \geq U$, and thus the coverage of Wald-type confidence intervals using $\widehat{V}(\A;\Y)$ will be also be at least proper levels.
\end{proof}

It therefore follows that Wald-type confidence intervals constructed using the conservative variance estimators derived in Section 3 yield asymptotic coverage at at least nominal levels.
\begin{corollary}
  Given Assumption \ref{assump:scaling}, then confidence intervals formed as $\overline X \pm z_{1-\alpha/2} \sqrt{\widehat{V}_1(\A^m)}$ have at least $100(1-\alpha)\%$ coverage for $\mu$ in large $n$.
\end{corollary}
\begin{corollary}
  Given Assumptions \ref{assump:scaling} and \ref{assump:homoskedastic}, then confidence intervals formed as $\overline X \pm z_{1-\alpha/2} \sqrt{\widehat{V}_2(\A^m;\Y)}$ or  $\overline X \pm z_{1-\alpha/2} \sqrt{\widehat{V}_2'(\A^m;\Y)}$ have at least $100(1-\alpha)\%$ coverage for $\mu$ in large $n$.
\end{corollary}
Proofs for Corollaries 3 and 4 follow directly from Corollaries 1 and 2 and Proposition 4.


\section{Application: HIV prevalence in a network study}

\label{sec:app}

The ``Sexual Acquisition and Transmission of HIV-Cooperative Agreement Program'' (SATH-CAP) surveyed $n=1022$ injection drug users, men who have sex with men, and their sexual partners in St. Petersburg, Russia from 2005 to 2008 \citep{iguchi2009simultaneous,niccolai2010high}. Subjects were recruited using a social network link-tracing procedure known as ``respondent-driven sampling'' (RDS) \citep{Heckathorn1997Respondent,Broadhead1998Harnessing}. Participants in an RDS study recruit other eligible subjects to whom they are connected within the target population social network. To preserve privacy, subjects do not report identifying information about their network alters; instead they report their \emph{degree} in the target population network. Researchers observe the social links along which recruitment takes place, and the degrees of recruited individuals. Each subject in the SATH-CAP study completed a demographic and behavioral quenstionnaire and also received a rapid HIV test. 

We treat the underlying social network as a dependency graph, denoted $\G=(\V,\mathcal{E})$ representing possible probabilistic dependencies between surveyed subjects' HIV status. Let the subgraph of recruitments be $\G_R=(\V_S,\mathcal{E}_R)$, a subgraph of $\G$; since only recruitment links in $\G$ were observed, the study design did not reveal the induced subgraph $\G_S$. For each subject $i\in\V_S$, we observe their reported total degree $d_i$ and their binary HIV status $X_i$. Let the vector of subjects' HIV status be $\X=(X_1,\ldots,X_n)$, and let the vector of their degrees be $\bd=(d_1,\ldots,d_n)$.  The study reveals $\Y=(\X,\bd,\G_R)$, as described in Assumption \ref{assump:obs}. 

The estimated HIV prevalence in the SATH-CAP study is $\hat{\mu}=\overline{X}=0.328$. Table \ref{tab:results} shows variance estimates and Wald-type 95\% asymptotic confidence intervals computed using the variance estimators described in this paper. The first column shows the na\"ive standard error estimate with corresponding confidence interval below. The second column gives results for the general case in which no assumptions are made about the variance of each $X_i$ (Section \ref{sec:general}). The third column gives results for the homoskedastic case in which $\var(X_i)$ is assumed to be equal to $\var(X_j)$ for $i\neq j$ (Section \ref{sec:homoskedastic}).  

The na\"ive confidence interval is the narrowest, and is equivalent to the case where the adjacency matrix $\A$ is diagonal.  Confidence intervals computed using the na\"ive estimator may dramatically understate the uncertainty in estimates of $\mu$, as the estimator ignores the possibility of dependence between units. Confidence intervals computed using  estimates $\widehat{V}_1$ in the general case are narrower than estimators $\widehat{V}_2$ computed under the homoskedasticity assumption.  The widest intervals are obtained from the bounds given by $\widehat{V}_2(\A^m)$ and $\widehat{V}_2'(\Y)$.  From Lemma \ref{lem:V2}, we see that $\widehat{V}_2(\A^m;\Y) = \widehat{V}_2'(\Y)$ because $\bd'=(d_1',\ldots,d_n')$ is the degree sequence of a compatible adjacency matrix in $\mathcal{A}$.

\begin{table}
\centering
\caption{Standard error estimates and 95\% asymptotic Wald-type confidence intervals for the population HIV prevalence $\mu$. }
\begin{tabular}{rccrccrc}
  \multicolumn{2}{c}{Na\"ive}  && \multicolumn{2}{c}{General} && \multicolumn{2}{c}{Homoskedastic} \\ 
  $\sqrt{\hat{\sigma}^2/n}$ & 0.0147      && $\sqrt{\widehat{V}_1(\A^m;\Y)}$ & 0.0563       && $\sqrt{\widehat{V}_2(\A^m;\Y)}$ & 0.0602 \\ 
        95\% CI:            & (0.299, 0.357) &&            95\%  CI:              & (0.217, 0.438) &&             95\% CI:             & (0.210, 0.446)  \\[0.5em]
                    &               &&                          &               && $\sqrt{\widehat{V}_2'(\Y)}$ & 0.0602 \\ 
                    &               &&                          &               &&        95\% CI:                  & (0.210, 0.446)
\end{tabular}
\label{tab:results}
\end{table}


\section{Discussion}

We have developed conservative estimators for the variance of the sample mean under partial observation of a dependency graph and assumptions about the variance of individual outcomes.  The variance estimation setting we address here is quite flexible, and can accommodate a wide variety of dependency and observation assumptions. For example, Assumption \ref{assump:obs}, which states that we observe $\Y=(\X,\bd,\G_R)$, can be weakened when $\G_R$ is completely unknown. In this case the constraint in the integer linear programs \eqref{eq:ilp1} and \eqref{eq:ilp2} becomes $\A \succeq \mathbf{0}$ where $\mathbf{0}$ is the $n\times n$ matrix of all zeros; this constraint is met for all adjacency matrices $\A$, so it becomes superfluous.  Alternatively, we may not have full knowledge of the degrees $\bd=(d_1,\ldots,d_n)$, and instead have only an upper bound $d_i^*$ for each $d_i$, or a global upper bound $d_i\le d^*$ for all $i=1,\ldots,n$. Conservative variance estimation in both of these cases can be achieved (by susbstituting $d_i^*$ or $d^*$ for $d_i$) with no change to the programs \eqref{eq:ilp1} and \eqref{eq:ilp2} or to the asymptotic results given here.  When no information about $\G_R$ or the degrees $\bd$ is available, setting every $d_i=d^*=n-1$ delivers a maximally conservative upper bound.  

We note here four extensions. (i) Upper bounds for the variance estimates can be obtained by solving a relaxed form of the programs \eqref{eq:ilp1} and \eqref{eq:ilp2}. By Proposition 3, using such upper bounds as a basis for conservative inference will also yield valid confidence intervals. In practice, the results obtained by modern optimization solvers will be tighter with a provably small optimality gap and thus will typically be preferable. (ii) It is possible to extend our results to obtain confidence intervals more generally for asymptotically linear estimators \citep[including regression estimators, e.g.,][]{cameron2015practitioner} using an empirical analogue of the variance of the influence function as the objective function. (iii) Our results facilitate conservative inference for causal estimands under interference between units \citep[e.g.,][]{tchetgen2010causal,liu2014large}, given interference that can be characterized by a constrained dependency graph. (iv) Given additional assumptions about the manner in which the units in the sample are drawn from a broader population, our results could be extended to facilitate confidence intervals for the mean of this broader population.


\section*{Acknowledgement}

Forrest W. Crawford was supported by NIH/NCATS grant KL2 TR000140 and NIMH grant P30MH062294.  
Jos\'{e} R. Zubizarreta acknowledges support from a grant from the Alfred P. Sloan Foundation.
We are grateful to Robert Heimer for helpful comments and for providing the SATH-CAP data, funded by NIH/NIDA grant U01DA017387. 
We also thank 
Daniel Bienstock,
Winston Lin,
Luke W. Miratrix,
Molly Offer-Westort,
Lilla Orr,
Cyrus Samii,
and
Jiacheng Wu
for valuable comments. We express special thanks to Sahand Negahban for important early discussions regarding the formulation of the problem.


\section*{Supplementary Material}


\appendix

\section{Formulation of the integer linear programs}

In order to solve the program (\ref{eq:ilp1}), let $\hat{v}_{ij}$ be the $ij$th element of the sample covariance matrix with $i = 1, ..., n$ and $j = 1, ..., n$.  
Since the sample covariance matrix is symmetric, we can focus on its upper triangular part and use the decision variable $a_{ij} = 1$ if $\hat{v}_{ij} \neq 0$, and 0 otherwise, for each $i < j$.
Based on these decision variables, the integer linear program (\ref{eq:ilp1}) can be written as 
\begin{equation*}
\begin{aligned}
& \underset{\boldsymbol{a}}{\text{maximize}}
&& \sum_{i = 1}^{n} \sum_{j = i+1}^{n} \hat{v}_{ij} a_{ij} \\
& \text{subject to}
&& \sum_{j = 1}^{i-1} a_{ji} + \sum_{j = i+1}^{n} a_{ij} \leq d_i, \; i = 1, ..., n, \\
&
&& a_{ij} \in \{ 0, 1\}, \; i = 1, ..., n, \; j = 1, ..., n, \; i < j, \\
\end{aligned}
\end{equation*}
where $d_i$ is the degree, and further simplified with the constraints $\A \succeq \A_R$ that make some of the decisions variables $a_{ij}$ automatically equal to one.
In order to solve the program (\ref{eq:ilp2}), let $\hat{v}_{ij} = 1$ for every $i = 1, ..., n$ and $j = 1, ..., n$.
These is are examples of the multidimensional knapsack problem studied in operations research (for a survey of this problem, see chapter 9 of \citealt{kellerer2004knapsack}).


\section{Statistical software implementation}

We implement this approach in the new statistical software package \texttt{depinf} for \texttt{R}.
\texttt{depinf} includes two basic functions: \texttt{depgraph}, for finding the adjacency matrix that maximizes the variance estimate of the mean given general constraints on the degree of dependence of the observations (these are problems (\ref{eq:ilp1}) and (\ref{eq:ilp2}) above), and \texttt{depvar} for calculating the variance estimates (\ref{eq:v1}) and (\ref{eq:v2}).
In both \texttt{depgraph} and \texttt{depinf}, we give the option to find an exact solution to (\ref{eq:ilp1}) and (\ref{eq:ilp2}) via integer programming, or an approximate solution to the relaxations of (\ref{eq:ilp1}) and (\ref{eq:ilp2}) via linear programming. 
Naturally, the running time of the approximate solution is lower, but it provides a more conservative variance estimate.
In order to solve (\ref{eq:ilp1}) and (\ref{eq:ilp2}), \texttt{depgraph} can use three different optimization solvers: CPLEX, GLPK and Gurobi. 
By default, \texttt{depgraph} uses GLPK, which can be downloaded from the \texttt{R} repository CRAN. 
To solve large instances of the problem exactly, we strongly recommend using either CPLEX or Gurobi, which are much faster but require a license and special installation. 
Between CPLEX or Gurobi, Gurobi is considerably easier to install. 
At the present \texttt{depinf} can be downloaded from \url{http://www.columbia.edu/~jz2313/} and will soon be available through CRAN.

\bibliography{depgraph}

\begin{thebibliography}{24}
\providecommand{\natexlab}[1]{#1}
\providecommand{\url}[1]{{#1}}
\providecommand{\urlprefix}{URL }
\expandafter\ifx\csname urlstyle\endcsname\relax
  \providecommand{\doi}[1]{DOI~\discretionary{}{}{}#1}\else
  \providecommand{\doi}{DOI~\discretionary{}{}{}\begingroup
  \urlstyle{rm}\Url}\fi
\providecommand{\eprint}[2][]{\url{#2}}

\bibitem[{Aronow and Samii(2013)}]{aronow2013estimating}
Aronow PM, Samii C (2013) Estimating average causal effects under interference
  between units. arXiv preprint arXiv:13056156

\bibitem[{Aronow et~al(2015)Aronow, Samii, and Assenova}]{aronow2015cluster}
Aronow PM, Samii C, Assenova VA (2015) Cluster-robust variance estimation for
  dyadic data. Political Analysis 23(4):564--577

\bibitem[{Baldi and Rinott(1989)}]{baldi1989normal}
Baldi P, Rinott Y (1989) On normal approximations of distributions in terms of
  dependency graphs. The Annals of Probability pp 1646--1650

\bibitem[{Balzer et~al(2015)Balzer, Petersen, and van~der
  Laan}]{balzer2015targeted}
Balzer LB, Petersen ML, van~der Laan MJ (2015) Targeted estimation and
  inference for the sample average treatment effect. bepress

\bibitem[{Bixby and Rothberg(2007)}]{bixby2007progress}
Bixby RE, Rothberg E (2007) Progress in computational mixed integer
  programming---a look back from the other side of the tipping point. Annals of
  Operations Research 149:37--41,
  \urlprefix\url{http://dx.doi.org/10.1007/s10479-006-0091-y}

\bibitem[{Broadhead et~al(1998)Broadhead, Heckathorn, Weakliem, Anthony,
  Madray, Mills, and Hughes}]{Broadhead1998Harnessing}
Broadhead RS, Heckathorn DD, Weakliem DL, Anthony DL, Madray H, Mills RJ,
  Hughes J (1998) Harnessing peer networks as an instrument for {AIDS}
  prevention: results from a peer-driven intervention. Public Health Reports
  113(Suppl 1):42

\bibitem[{Cameron and Miller(2015)}]{cameron2015practitioner}
Cameron AC, Miller DL (2015) A practitioner's guide to cluster-robust
  inference. Journal of Human Resources 50(2):317--372

\bibitem[{Chen et~al(2004)Chen, Shao et~al}]{chen2004normal}
Chen LH, Shao QM, et~al (2004) Normal approximation under local dependence. The
  Annals of Probability 32(3):1985--2028

\bibitem[{Conley(1999)}]{conley1999gmm}
Conley TG (1999) {GMM} estimation with cross sectional dependence. Journal of
  Econometrics 92(1):1--45

\bibitem[{Crawford(2016)}]{crawford2016graphical}
Crawford FW (2016) The graphical structure of respondent-driven sampling.
  Sociological Methodology In press,
  \urlprefix\url{http://arxiv.org/abs/1406.0721}

\bibitem[{Heckathorn(1997)}]{Heckathorn1997Respondent}
Heckathorn DD (1997) Respondent-driven sampling: a new approach to the study of
  hidden populations. Social Problems 44(2):174--199

\bibitem[{Iguchi et~al(2009)Iguchi, Ober, Berry, Fain, Heckathorn, Gorbach,
  Heimer, Kozlov, Ouellet, Shoptaw, and Zule}]{iguchi2009simultaneous}
Iguchi MY, Ober AJ, Berry SH, Fain T, Heckathorn DD, Gorbach PM, Heimer R,
  Kozlov A, Ouellet LJ, Shoptaw S, Zule WA (2009) Simultaneous recruitment of
  drug users and men who have sex with men in the {United States} and {Russia}
  using respondent-driven sampling: sampling methods and implications. Journal
  of Urban Health 86(1):5--31

\bibitem[{Kellerer et~al(2004{\natexlab{a}})Kellerer, Pferschy, and
  Pisinger}]{kellerer2004introduction}
Kellerer H, Pferschy U, Pisinger D (2004{\natexlab{a}}) Introduction to
  NP-Completeness of knapsack problems. Springer

\bibitem[{Kellerer et~al(2004{\natexlab{b}})Kellerer, Pferschy, and
  Pisinger}]{kellerer2004knapsack}
Kellerer H, Pferschy U, Pisinger D (2004{\natexlab{b}}) Knapsack problems.
  Springer

\bibitem[{van~der Laan et~al(2013)van~der Laan, Hubbard, and
  Pajouh}]{vanderlaan2013statistical}
van~der Laan MJ, Hubbard AE, Pajouh SK (2013) Statistical inference for data
  adaptive target parameters. bepress

\bibitem[{Liang and Zeger(1986)}]{liang1986longitudinal}
Liang KY, Zeger SL (1986) Longitudinal data analysis using generalized linear
  models. Biometrika pp 13--22

\bibitem[{Linderoth and Lodi(2010)}]{linderoth2012milp}
Linderoth JT, Lodi A (2010) {MILP} software. In: Cochran JJ, Cox LA, Keskinocak
  P, Kharoufeh JP, Smith JC (eds) Wiley Encyclopedia of Operations Research and
  Management Science, Wiley, \doi{10.1002/9780470400531.eorms0524},
  \urlprefix\url{http://dx.doi.org/10.1002/9780470400531.eorms0524}

\bibitem[{Liu and Hudgens(2014)}]{liu2014large}
Liu L, Hudgens MG (2014) Large sample randomization inference of causal effects
  in the presence of interference. Journal of the American Statistical
  Association 109(505):288--301

\bibitem[{Nemhauser(2013)}]{nemhauser2013impact}
Nemhauser GL (2013) Integer programming: Global impact. EURO INFORMS July 2013

\bibitem[{Niccolai et~al(2010)Niccolai, Toussova, Verevochkin, Barbour, Heimer,
  and Kozlov}]{niccolai2010high}
Niccolai LM, Toussova OV, Verevochkin SV, Barbour R, Heimer R, Kozlov AP (2010)
  High {HIV} prevalence, suboptimal {HIV} testing, and low knowledge of
  {HIV}-positive serostatus among injection drug users in {St. Petersburg,
  Russia}. AIDS and Behavior 14:932--941

\bibitem[{Ogburn and VanderWeele(2014)}]{ogburn2014vaccines}
Ogburn EL, VanderWeele TJ (2014) Vaccines, contagion, and social networks.
  arXiv preprint arXiv:14031241

\bibitem[{Tabord-Meehan(2015)}]{tabord2015inference}
Tabord-Meehan M (2015) Inference with dyadic data: Asymptotic behavior of the
  dyadic-robust t-statistic. arXiv preprint arXiv:151007074

\bibitem[{Tchetgen and VanderWeele(2010)}]{tchetgen2010causal}
Tchetgen EJT, VanderWeele TJ (2010) On causal inference in the presence of
  interference. Statistical Methods in Medical Research p 0962280210386779

\bibitem[{White(2014)}]{white2014asymptotic}
White H (2014) Asymptotic Theory for Econometricians. Academic press

\end{thebibliography}
\bibliographystyle{spbasic}

\end{document}